\documentclass[12pt]{amsart}

\usepackage{amsfonts,amsthm,amsmath,amssymb,amscd,mathrsfs,tikz}
\usepackage{graphics}
\usepackage{indentfirst}
\usepackage{cite}
\usepackage{latexsym}
\usepackage[dvips]{epsfig}
\usepackage{color}
\usepackage{enumerate}
\usepackage{verbatim}
\usetikzlibrary{patterns}

\usepackage{hyperref}

\newtheorem{theorem}{Theorem}[section]

\newtheorem{lemma}[theorem]{Lemma}

\renewcommand{\div}{{\rm div\, }}
\newcommand{\p}{\partial}
\newcommand{\R}{\mathbb{R}}
\renewcommand{\r}{\rho}
\newcommand{\te}{\theta}
\newcommand{\al}{\alpha}
\newcommand{\vp}{\varphi}
\newcommand{\ve}{\varepsilon}
\newcommand{\nb}{\nabla}
\newcommand{\B}{\mathbf{B}}
\newcommand{\bu}{\mathbf{u}}
\newcommand{\bn}{\mathbf{n}}
\def\e{\mathbf{e}}
\def\bv{\mathbf{v}}
\def\X{\mathbf{X}}
\def\Y{\mathbf{Y}}
\def\bU{\mathbf{U}}
\def\bt{\boldsymbol{\tau}}

\begin{document}

\title[Axisymmetric self-similar solutions to the MHD equations]{Axisymmetric self-similar solutions to the MHD equations without magnetic diffusion}

\author{Shaoheng Zhang} 
\address{School of Mathematical Sciences, Soochow University, Suzhou, 215006, China}
\email{20234007008@stu.suda.edu.cn}


\begin{abstract}
We study the axisymmetric self-similar solutions $(\mathbf{u},\mathbf{B})$ to the stationary MHD equations without magnetic diffusion, where $\mathbf{B}$ has only the swirl component.
Our first result states that in $\mathbb{R}^3\setminus\{0\}$, $\mathbf{u}$ is a Landau solution and $\mathbf{B}=0$.
Our second result proves the triviality of axisymmetric self-similar solutions in the half-space $\mathbb{R}^3_+$ with the no-slip boundary condition or the Navier slip boundary condition.
\end{abstract}

\maketitle

\section{Introduction}

In this paper, we study the stationary incompressible magnetohydrodynamics (MHD) equations without magnetic diffusion
\begin{align}
\label{MHD}
    \begin{cases}
    -\Delta \bu+(\bu \cdot \nb)\bu-(\B \cdot \nb)\B+\nb p =0,\\
    (\bu \cdot \nb) \B =(\B \cdot \nb) \bu,\\
    \div  \bu= \div  \B=0,
    \end{cases}
\end{align}    
in $\Omega \subseteq \R^3$.  Here, $\bu$ is the velocity field of the fluid, $\B$ is the magnetic field, and $p$ is the scalar pressure. The MHD equations describe the motion of electrically conducting fluids, such as plasmas, and are a coupled system of the Navier-Stokes equations and the Maxwell's equations. 

The MHD equations \eqref{MHD} are invariant under the following scaling:
\begin{align*}
    \bu(x)&\to \bu_{\lambda}(x)=\lambda \bu(\lambda x),\\
    \B(x)&\to \B_{\lambda}(x)=\lambda \B(\lambda x),\\
    p(x)&\to p_{\lambda}(x)=\lambda^2 p(\lambda x),
\end{align*}
where $\lambda>0$.
The solution $(\bu,\B,p)$ is called \emph{self-similar} if 
\[
\bu(x)=\bu_{\lambda}(x), \quad \B(x)=\B_{\lambda}(x),\quad p(x)=p_{\lambda}(x)
\] 
for any $\lambda>0$. In this paper, we aim to characterize self-similar solutions to \eqref{MHD} under the assumption of axisymmetry.

When $\B\equiv0$, the MHD equations \eqref{MHD} reduce to the Navier-Stokes equations.
The well-known self-similar solutions to the Navier-Stokes equations in $\R^3\setminus\{0\}$ are the Landau solutions (cf. \cite{CK04,Landau44,Squire51,TX98}).
The Landau solutions are axisymmetric without swirl.
The explicit formula of Landau solutions can be
found in many textbooks (see, e.g., Section 8.2 in \cite{Tsai18}), see also Section \ref{Sect2.2}.
Tian and Xin \cite{TX98}, Cannone and Karch \cite{CK04} proved that a self-similar solution to the Navier-Stokes equations in $\R^3\setminus\{0\}$ is a Landau solution if the solution is axisymmetric.
Recently, \v{S}ver\'ak \cite{Sverak11} proved the assertion without the assumption of axissymmetry.
Kang, Miura and Tsai \cite{KMT18} proved the triviality of axisymmetric self-similar solutions to the Navier-Stokes equations in $\R^3_+$ with the no-slip boundary condition or the Navier slip boundary condition.

Regarding the MHD equations, Zhang, Wang, and Wang investigated the stationary incompressible MHD equations with magnetic diffusion in \cite{ZWW25}. 
Assume that $(\bu,\B)$ is a smooth axisymmetric self-similar solution in $\R^3\setminus\{0\}$ and $\B=B^{\te}(\r,\vp)\e_{\te}$, it was proved that $\bu$ is a Landau solution and $\B$ vanishes provided $|\bu(x)|\le \ve|x|^{-1}$ for some small $\ve>0$.

Let $\bu$ be a Landau solution and $\B = 0$, it can be verified that $(\bu,\B)$ is an axisymmetric self-similar solution to \eqref{MHD} in $\R^3\setminus\{0\}$. 
Our first result shows that the converse statement is also true if $\B$ has only the swirl component.

\begin{theorem}
\label{thm1}
Let $(\bu,\B)$ be a smooth axisymmetric self-similar solution to \eqref{MHD} in $\R^3\setminus\{0\}$.
Assume that $\B=\frac{B(\vp)}{\r}\e_\te$, where $(\r,\te,\vp)$ are the spherical coordinates defined in Section \ref{Sect2.1}.
Then $\bu$ is a Landau solution and $\B=0$.
\end{theorem}

Our second result is the following theorem, which proves the triviality of the axisymmetric self-similar solutions in the half-space $\mathbb{R}^3_+$ with the no-slip boundary condition or the Navier slip boundary condition.

\begin{theorem}
\label{thm2}
Let $(\bu,\B)$ be a smooth axisymmetric self-similar solution to \eqref{MHD} in $\R^3_+$, satisfying the no-slip boundary condition 
\begin{align}\label{noslipbc}
    \bu=0,  \quad\B\cdot \bn= 0, \quad &\text{on  }\ \p\R^3_+\setminus\{0\},
\end{align}
or the Navier slip boundary condition
\begin{equation}
\begin{aligned}\label{navierslipbc}
    \bu\cdot \bn=0, \quad (\mathbb{D}\bu \cdot \bn)\cdot \bt=0  , \quad\B\cdot \bn= 0, \quad &\text{on  }\ \p\R^3_+\setminus\{0\},
\end{aligned}
\end{equation}
where $\bn$ and $\bt$ are the unit outward normal and tangent vectors on $\p\R^3_+$, respectively, and $\mathbb{D}\bu=\frac{1}{2}(\nb \bu+\nb \bu^T)$  is the  strain tensor.
Assume that $\B=\frac{B(\vp)}{\r}\e_\te$.
Then $\bu=\B= 0$.
\end{theorem}

The paper is organized as follows. In Section 2, we recall basic definitions and introduce Landau solutions. Section 3 is devoted to the proof of Theorems \ref{thm1} and \ref{thm2}.

\section{Preliminaries}

\subsection{Axisymmetric solutions}\label{Sect2.1}
In this subsection, we introduce the axisymmetric vector fields.
Recall that the spherical coordinates of $\R^3$ are $(\r,\te,\vp)$ with 
\begin{equation*}
  (x_1, x_2, x_3) =(\r\sin\vp\cos \te,\r\sin\vp\sin \te, \r \cos \vp), 
\end{equation*} 
and the corresponding orthonormal basis vectors $\e_{\r},\e_{\te},\e_{\vp}$ are
\begin{align} \label{sphericalbasis}
\e_{\r}=\frac{x}{\r}, \quad \e_{\te}=(-\sin \te, \cos \te,0), \quad \e_{\vp}=\e_{\te}\times \e_{\rho}.
\end{align}
In this coordinate system, a vector field $\bv$ can be expressed as
\begin{equation*}
    \bv = v^\r(\r,\te,\vp) \e_\r + v^\te(\r,\te,\vp) \e_\te + v^\vp(\r,\te,\vp) \e_\vp,  
\end{equation*}
with scalar components $v^\r,v^\te,v^\vp$, and we call $v^\te$ its swirl component.
A vector field $\bv$ is called \emph{axisymmetric}, 
if it is of the form
\[
\bv = v^\r(\r,\vp) \e_\r + v^\te(\r,\vp) \e_\te + v^\vp(\r,\vp) \e_\vp. 
\]

\subsection{Landau solutions to the Navier-Stokes equations}
\label{Sect2.2}
For each $b\in \R^3$, there exists a unique
$(-1)$-homogeneous solution $\bU^b$ to the Navier-Stokes equations, together with an associated $(-2)$-homogeneous pressure $P^b$, such that $\bU^b, P^b$ are smooth in $\R^3 \setminus \{0\}$ and satisfy
\begin{align*}
-\Delta \bU^b + \div(\bU^b \otimes \bU^b) + \nb P^b =b\delta, 
\quad \div \bU^b = 0,
\end{align*}
in the distributional sense in $\R^3$, where $\delta$ denotes the Dirac function. 
If $b=(0,0,\beta)$ with $\beta \ne 0$, the solutions $\bU^b$ and $P^b$ are given explicitly by
\begin{align}\label{Landausolutions}
    \bU^b=\frac{2}{\rho} 
    \left[
    \left(\frac{a^2-1}{(a-\cos \vp)^2}-1\right)\e_{\rho}
    +\frac{-\sin \vp}{a-\cos \vp}\e_{\vp}
    \right],
    \quad
    P^b=\frac{4(a \cos \vp-1)}{\rho^2 (a-\cos \vp)^2},
\end{align}
where $\e_\r,\e_{\te},\e_{\vp}$ are defined in \eqref{sphericalbasis} and $|a|>1$.
The relationship between $\beta \neq 0$ and $|a|>1$ is as follows
\[
\beta = 16 \pi 
\left[a + \frac 12 a^2 \log \left(\frac {a-1}{a+1}\right) + \frac{4a}{3(a^2-1)}\right].
\]
The corresponding Landau solution for a general $b$ can be obtained by rotation.

\section{Proof of the main Results}

In this section, we will prove Theorems \ref{thm1}--\ref{thm2}. 
The strategy follows that for the derivation of Landau solutions (cf. \cite[Theorem 8.1]{Tsai18}).
Under the assumption of axisymmetry, the MHD equations for a self-similar solution reduce to a system of ODEs, which can be solved explicitly.


\begin{proof}[\textbf{Proof of Theorem \ref{thm1}}]
In spherical coordinates $(\r,\vp,\te)$, the axisymmetric self-similar solution $(\bu,\B,p)$ is of the form 
\[
\bu=\frac{f(\vp)}{\r}\e_\r+\frac{g(\vp)}{\r}\e_\vp+\frac{h(\vp)}{\r}\e_\te,\quad 
\B=\frac{B(\vp)}{\r}\e_{\te},\quad 
p=\frac{P(\vp)}{\r^2}.
\]
Using Appendix 2 in \cite{Batchelor99}, the system \eqref{MHD} reduces to the following ODEs  (see Lemma \ref{lemmaA2} in Appendix A for an alternative derivation using the Riemannian connection)
\begin{align}\label{theorem1-1}
    f''+f' \cot \vp = f'g-(f^2+g^2+h^2)+B^2-2P,
\end{align}
\begin{align}\label{theorem1-2}
    f'=g'g-h^2 \cot \vp +B^2 \cot \vp +P',
\end{align}
\begin{align}\label{theorem1-3}
    (h'+h \cot \vp)'=g(h'+h \cot \vp),
\end{align}
\begin{align}\label{theorem1-4}
    0=gB'-gB \cot \vp-2Bf,
\end{align}
\begin{align}\label{theorem1-5}
    0=f+g'+g \cot \vp,
\end{align}
for $\vp \in (0, \pi)$, together with the following boundary conditions
\begin{align}
\label{theorem1-6}
    f'(0)=g(0)=h(0)=B(0)=f'(\pi)=g(\pi)=h(\pi)=B(\pi)=0,
\end{align}
due to symmetry and regularity of $\bu$ and $\B$.

Define  $H(\vp):=h'(\vp)+h(\vp)\cot \vp=\frac{(h \sin\vp)'}{\sin \vp}$. Equation \eqref{theorem1-3} reduces to $H'=gH$. 
There exists $\vp_1\in (0,\pi)$ such that $H(\vp_1)=0$, since 
\[
\int_0^{\pi} H(\vp)\sin \vp\ \mathrm{d}\vp 
= h(\vp)\sin \vp \big|_0^{\pi} =0.
\]
Then by uniqueness of solutions, $H$ is identically zero. Therefore $h(\vp)\sin \vp$ is a constant, which must be $0$ due to the boundary condition. 
\textbf{Hence, we conclude that $\boldsymbol{h(\vp)\equiv0}$}. 

Substituting \eqref{theorem1-5} into \eqref{theorem1-4}, we get
\[
0=(Bg)'+Bg'+Bg\cot \vp.
\]
Multiplying the above equation by $g\sin \vp$ gives
\[
0=(Bg^2\sin \vp)'.
\]
The boundary conditions $g(0)=B(0)=0$ imply that 
\begin{align}\label{theorem1-7}
Bg^2=0, \quad \vp\in  (0,\pi).
\end{align}
\textbf{We claim that $\boldsymbol{ B(\vp)\equiv0}$ }. Otherwise, there exists $\vp_2\in (0,\pi)$ such that $B(\vp_2)\neq0$. 
By the regularity of $B$ and the boundary condition, there exists a maximal interval $(\alpha,\beta)\subseteq (0,\pi)$ such that $B\neq0$ on $(\al,\beta)$ and $B(\alpha)=B(\beta)=0$. Equations \eqref{theorem1-7} and \eqref{theorem1-5} show $g=0$ and $f=0$ on $(\al,\beta)$. Then we have
\[
B^2-2P=0, \quad B^2 \cot\vp+P'=0,\quad \vp \in (\al,\beta),
\]
by \eqref{theorem1-1} and \eqref{theorem1-2}. 
Direct calculation shows $P=\frac{C_1}{\sin^2\vp}$ on $(\al,\beta)$ for some constant $C_1$. Since $P(\alpha)=\frac{B(\al)^2}{2}=0$, we have $C_1=0$. Hence, $B=0$ on $(\al,\beta)$, which gives the contradiction.

The derivation of $f$ and $g$ is the same as that of Landau solutions (see, e.g., \cite[Theorem 8.1]{Tsai18}).
Integrating \eqref{theorem1-2} gives
\begin{align}\label{theorem1-8}
    f=\frac{g^2}{2}+P+C_2
\end{align}
for some $C_2$. Substituting \eqref{theorem1-8} into \eqref{theorem1-1}, we get
\begin{align}\label{theorem1-9}
    f''+f'\cot\vp=gf'-f^2-2f+2C_2.
\end{align}
Set $J(\vp):=f(\vp)\sin \vp$, and $K(\vp):=g(\vp)\sin \vp$. Equation \eqref{theorem1-5} yields $J=-K'$. Thus \eqref{theorem1-9} becomes
\[
(f'\sin \vp)'=(Kf)'+2K'+2C_2\sin \vp.
\]
Integrating the above equation, we get
\[
f'\sin \vp=Kf+2K-2C_2\cos \vp +C_3.
\]
By the boundary conditions \eqref{theorem1-6}, we get $0=\pm2C_2+C_3$. Thus $C_2=C_3=0$ and 
\begin{align}\label{theorem1-10}
    f'\sin \vp=Kf+2K.
\end{align}
Let $L(t)=K(\vp)$ with $t=\cos \vp$.
Direct computation gives $K'(\vp)=-L'(t)\sin \vp$, which yields $f(\vp)=L'(t)$ by using $J=-K'$.
Therefore we have
\begin{align*}
   &(1-t^2)L''+2L+LL'=0, \quad \text{for} \quad t\in (-1,1),\\
    &L(-1)=L(1)=0.
\end{align*}
Set $L(t)=(1-t^2)w(t)$. Then $w$ satisfies
\[
w'+\frac{w^2}{2}=0.
\]
Thus $w(t)=\frac{2}{t-a}$ for some $|a|>1(a\in \R)$, and 
\[
L(t)=\frac{2(1-t^2)}{t-a}=K(\vp)=\frac{2\sin^2\vp}{\cos \vp-a}.
\]
Thus we have
\[
f(\vp)=\frac{2(a^2-1)}{(a-\cos\vp)^2}-2, \quad g(\vp)=\frac{2\sin \vp}{\cos \vp-a}, \quad
h(\vp)= B(\vp)=0.
\]
Equation \eqref{theorem1-8} shows $P=f-\frac {g^2}{2}$.
Hence, the solution $\bu$ is exactly a Landau solution, as introduced in Section \ref{Sect2.2}.
\end{proof}

\begin{proof}[\textbf{Proof of Theorem \ref{thm2}}]
The axisymmetric self-similar solution $(\bu,\B,p)$ is of the form 
\[
\bu=\frac{f(\vp)}{\r}\e_\r+\frac{g(\vp)}{\r}\e_\vp+\frac{h(\vp)}{\r}\e_\te,\quad 
\B=\frac{B(\vp)}{\r}\e_{\te},\quad 
p=\frac{P(\vp)}{\r^2}.
\]
Then the system \eqref{MHD} becomes
\begin{align}\label{theorem2-1}
    f''+f' \cot \vp = f'g-(f^2+g^2+h^2)+B^2-2P,
\end{align}
\begin{align}\label{theorem2-2}
    f'=g'g-h^2 \cot \vp +B^2 \cot \vp +P',
\end{align}
\begin{align}\label{theorem2-3}
    (h'+h \cot \vp)'=g(h'+h \cot \vp),
\end{align}
\begin{align}\label{theorem2-4}
    0=gB'-gB \cot \vp-2Bf,
\end{align}
\begin{align}\label{theorem2-5}
    0=f+g'+g \cot \vp,
\end{align}
for $\vp \in (0,\frac\pi2)$. 
Due to symmetry, we have
\begin{align}
\label{theorem2-6}
    f'(0)=g(0)=h(0)=B(0)=0.
\end{align}
Note that the boundary conditions \eqref{noslipbc} and \eqref{navierslipbc} become
\begin{align}\label{bdc1}
f(\frac\pi2)=g(\frac\pi2)=h(\frac\pi2)=0,
\end{align}
and 
\begin{align}\label{bdc2}
f'(\frac\pi2)=g(\frac\pi2)=h'(\frac\pi2)=0,
\end{align}
respectively.

Define  $H(\vp):=h'+h\cot \vp=\frac{(h\sin\vp)'}{\sin \vp}$. Then equation \eqref{theorem2-3} becomes $H'=gH$. 
In the case of the no-slip boundary condition, equations \eqref{theorem2-6} and \eqref{bdc1} imply that
\[
\int_0^{\frac{\pi}{2}}H\sin\vp\, \mathrm{d}\vp
=\int_0^{\frac{\pi}{2}}(h\sin\vp)'\, \mathrm{d}\vp=0,
\]
which gives the existence of $\vp_3\in (0,\frac\pi2)$ such that $H(\vp_3)=0$.
In the case of the Navier slip boundary condition,
\[
H(\frac \pi 2)=h'(\frac \pi 2)=0.
\]
In both cases, $H$ is identically zero because of the uniqueness of solutions to an ODE. 
Hence, $h(\vp)\sin \vp$ is a constant, which must be $0$ by the boundary condition $h(0)=0$. 
\textbf{We conclude that $\boldsymbol{h(\vp)\equiv0}$}. 

Equations \eqref{theorem2-4} and \eqref{theorem2-5} show
\[
(Bg^2\sin \vp)'=0.
\]
Combining with $g(0)=B(0)=0$ gives
\[
Bg^2=0, \quad  \vp \in (0,\frac\pi2).
\]
Similarly, \textbf{we claim $\boldsymbol{B(\vp)\equiv0}$}. Otherwise there exists $\vp_4\in (0,\frac\pi2]$ such that $B(\vp_4)\ne0$. Let $(\al,\beta)$ be the maximal interval containing $\varphi_4$ where $B \neq 0$ and $B(\alpha) = 0$.
Following the same arguments above, equations \eqref{theorem2-1},\eqref{theorem2-2} and \eqref{theorem2-5} imply
\[
B^2-2P=0,\quad B^2\cot\vp+P'=0, \quad \vp\in(\al,\beta).
\]
The boundary condition $P(\al)=0$ and the ODE $P'+2P\cot\vp=0$ yield $P=0$. Then $B=0$ on $(\al,\beta)$, which gives the contradiction. 

Therefore the system becomes
\begin{align*}
    -\Delta \bu+(\bu \cdot \nb)\bu+\nb p =0,\ \div  \bu=0, \quad \text{in} \quad \R^3_+,
\end{align*} 
satisfying the no-slip boundary condition
\begin{align*}
 \bu=0,  \quad &\text{on} \quad \p\R^3_+\setminus\{0\},
\end{align*}
or the Navier slip boundary condition
\begin{align*}
    \bu\cdot \bn=0, \quad (\mathbb{D}\bu \cdot \bn)\cdot \bt=0  , \quad &\text{on  }\ \p\R^3_+\setminus\{0\}.
\end{align*}
According to \cite[Theorem 5.1]{KMT18}, $\bu=0$, which completes the proof of Theorem \ref{thm2}.
\end{proof}

\section*{Acknowledgements}
The author would like to thank Professors Yun Wang and Kui Wang for helpful discussions, encouragements, and supports.
The author was supported by the Postgraduate Research \& Practice Innovation Program of Jiangsu Province via grant KYCX24\_3285.

\appendix

\section{}
In this section, we give an alternative derivation of equations \eqref{theorem1-1}--\eqref{theorem1-5} using the Riemannian connection.
Let $\mathfrak{X}(\R^3)$ be the set of smooth vector fields on $\R^3$. 
Then the Riemannian connection $\nb:\mathfrak{X}(\R^3)\times \mathfrak{X}(\R^3) \to \mathfrak{X}(\R^3)$, denoted by $(\X,\Y)\mapsto \nb_\X \Y$, satisfies the following properties:
\begin{enumerate}[(i)]
    \item $\nb_{(f_1\X_1+f_2\X_2)}\Y=f_1(\nb_{\X_1}\Y)+f_2(\nb_{\X_2}\Y)$;
    \item  $\nb_\X{(f_1\Y_1+f_2\Y_2)}
    =\X(f_1)\Y_1+f_1 (\nb_{\X}\Y_1)+\X(f_2)\Y_2+f_2 (\nb_{\X}\Y_2)$;
    \item $\X\langle \Y_1,\Y_2\rangle=\langle \nb_\X \Y_1,\Y_2 \rangle+\langle \Y_1,\nb_\X \Y_2 \rangle$;
    \item $\nb_\X\Y-\nb_\Y \X=[\X,\Y]$;
\end{enumerate}
for all $\X,\Y,\X_1,\X_2,\Y_1,\Y_2\in \mathfrak{X}(\R^3)$ and $f_1,f_2\in C^{\infty}(\R^3)$. Here $\langle\cdot ,\cdot  \rangle$ is the standard Euclidean metric, and  $[\cdot ,\cdot ]$ is the Lie bracket. 

Using the notations in Section \ref{Sect2.1}, the canonical Euclidean metric is $g_{\R^3}=\mathrm{d}\r^2+\r^2\ \mathrm{d} \vp^2+\r^2 \sin^2\vp\ \mathrm{d}\te^2$, and $\e_{\r}=\frac{\p}{\p \r}, \e_{\vp}=\frac{1}{\r}\frac{\p}{\p \vp}, \e_{\te}=\frac{1}{\r \sin \vp}\frac{\p}{\p \te}$.

\begin{lemma}\label{lemmaA}
Let $\nb$ be the Riemannian connection on $\R^3$. For any $\bv=v^\r(\r,\vp,\te)\e_\r+v^\vp(\r,\vp,\te)\e_\vp+v^\te(\r,\vp,\te)\e_\te\in \mathfrak{X}(\R^3)$, we have
\begin{align}\label{lemmaA1-1}
\left\{
\begin{aligned}
&\nb_{\e_\r}\e_{\r}=\nb_{\e_\r}\e_{\vp}=\nb_{\e_\r}\e_{\te}=0,\\
&\nb_{\e_\vp}\e_{\r}=\frac1\r\, \e_{\vp},\  
\nb_{\e_\vp}\e_{\vp}=-\frac1\r\, \e_{\r},\ 
\nb_{\e_\vp}\e_{\te}=0,\\
&\nb_{\e_\te}\e_{\r}=\frac1\r\, \e_{\te},\
\nb_{\e_\te}\e_{\vp}= \frac{\cot \vp}{\r}\, \e_{\te},\
\nb_{\e_\te}\e_{\te}=
-\frac1\r\, \e_{\r}-\frac{\cot \vp}{\r}\, \e_{\vp},
\end{aligned}
\right.
\end{align}
and
\begin{align}\label{lemmaA1-2}
\left\{
\begin{aligned}
\nb_{\e_\r}\bv
=&\frac{\p v^\r}{\p \r}\e_\r
+\frac{\p v^\vp}{\p \r}\e_\vp
+\frac{\p v^\te}{\p \r}\e_\te,\\
\nb_{\e_\vp}\bv
=&\frac1\r \left( \frac{\p v^\r}{\p \vp}-v^\vp \right) \e_\r
+\frac1\r \left( v^\r+ \frac{\p v^\vp}{\p \vp} \right) \e_\vp
+\frac1\r \frac{\p v^\te}{\p \vp}\e_\te,\\
\nb_{\e_\te}\bv
=&\frac1\r \left( \frac{1}{\sin \vp}\frac{\p v^\r}{\p \te}-v^\te \right) \e_\r
+\frac1\r \left( \frac{1}{\sin \vp}\frac{\p v^\vp}{\p \te} -\cot \vp v^\te \right) \e_\vp\\
&+\frac1\r \left(v^{\r}+ \cot \vp v^{\vp} + \frac{1}{\sin \vp} \frac{\p v^\te}{\p \te} \right)\e_\te.
\end{aligned}
\right.
\end{align}
Here $\e_\r,\e_\te$ and $\e_\vp$ are defined in \eqref{sphericalbasis}.
\end{lemma}

\begin{proof}
We first prove equation \eqref{lemmaA1-1}. For any $\Y=(Y^1,Y^2,Y^3)=\sum_{i=1}^3Y^i \frac{\p}{\p x_i}\in \mathfrak{X}(\R^3)$, we have
\begin{align}\label{lemmaA1-3}
    \nb_\X \Y=(\X(Y^1),\X(Y^2),\X(Y^3)).
\end{align}
Here we used the properties (i) and (ii) of the Riemannian connection and the fact that $\nb_{\frac{\p}{\p x_i}} \frac{\p}{\p x_j}=0$ on $\R^3$ (see, e.g., \cite[p. 56]{dC92}). The proof of \eqref{lemmaA1-1} follows directly from \eqref{lemmaA1-3}, and we only prove the case $\nb_{\e_{\te}} \e_{\kappa}\ (\kappa=\r,\vp,\te)$. Using \eqref{lemmaA1-3}, we have
\begin{align*}
    \nb_{\e_\te}\e_\r
    =&\frac{1}{\r \sin \vp}
    \left(
    \p_{\te}(\sin \vp \cos \te), 
    \p_{\te}(\sin \vp \sin \te), 
    \p_{\te}(\cos \vp) 
    \right)
    =\frac{1}{\r}\, \e_\te, \\
    \nb_{\e_\te}\e_\vp
    =&\frac{1}{\r \sin \vp}
    \left(
    \p_{\te}(\cos \vp \cos \te),
    \p_{\te}(\cos \vp \sin \te),
    \p_{\te}(-\sin \vp)
    \right)
    =\frac{1}{\r}\cot \vp\, \e_\te, \\
    \nb_{\e_\te}\e_\te
    =&\frac{1}{\r \sin \vp}
    \left(
    \p_{\te}(-\sin \te), \p_{\te}(\cos \te), 0
    \right)
    = -\frac{1}{\r}\, \e_\r-\frac{1}{\r} \cot \vp\, \e_\vp .
\end{align*}
The proof of equation \eqref{lemmaA1-2} relies on \eqref{lemmaA1-1}, and we just give a proof for the case $\nb_{\e_{\te}}\bv$. In this case, direct calculation gives
\begin{align*}
    \nb_{\e_{\te}} \bv 
    =& (\nb_{\e_{\te}} v^{\r}) \e_{\r} + v^{\r}(\nb_{\e_{\te}} \e_{\r})
    + (\nb_{\e_{\te}} v^{\vp}) \e_{\vp} + v^{\vp}(\nb_{\e_{\te}} \e_{\vp})
    + (\nb_{\e_{\te}} v^{\te}) \e_{\te} + v^{\te}(\nb_{\e_{\te}} \e_{\te})\\
    =& \frac1\r \left( \frac{1}{\sin \vp}\frac{\p v^\r}{\p \te}-v^\te \right) \e_\r
    +\frac1\r \left( \frac{1}{\sin \vp}\frac{\p v^\vp}{\p \te} -\cot \vp v^\te \right) \e_\vp\\
    &+\frac1\r \left(v^{\r}+ \cot \vp v^{\vp} + \frac{1}{\sin \vp} \frac{\p v^\te}{\p \te} \right)\e_\te.
\end{align*}
The first equality follows from the property (ii) of the Riemannian connection, and we used \eqref{lemmaA1-1} in the second equality.
\end{proof}

The next lemma gives the Laplacian, divergence, and other quantities related to the MHD equations in spherical coordinates. First, we recall some definitions. If $\bv$ is a vector field and $p$ is a scalar function, the Laplacian and the divergence of $\bv$, and the gradient of $p$ are given by (cf. \cite[pp. 12--18]{CLN06})
\begin{align} \label{geometricdefi}
\Delta \bv=\sum_i \left( \nb_{\e_i} \nb_{\e_i} \bv -\nb_{\nb_{\e_i} \e_i} \bv  \right), \quad
\div \bv=\sum_i \langle \nb_{\e_i} \bv,\e_i  \rangle,\quad 
\nb p= \sum_i \e_i(p)\, \e_i,        
\end{align}
where $\{\e_i\}$ is an orthonormal basis and $\langle\cdot,\cdot \rangle$ denotes the standard Euclidean metric.

\begin{lemma}
\label{lemmaA2}
Let $\displaystyle \bu=\frac{f(\vp)}{\r}\e_\r+\frac{g(\vp)}{\r}\e_\vp+\frac{h(\vp)}{\r}\e_\te$ with $f,g,h\in C^2([0,\pi])$.
\begin{enumerate}[(1)]
\item The Laplacian and the divergence of $\bu$ are given by
\begin{align*}
\Delta \bu=&\frac{1}{\r^3}
\Big( \big(f''+f'\cot\vp-2(f+g'+g\cot\vp)\big)\,\e_\r
+(2f+g'+g\cot\vp)'\,\e_{\vp}\\
&+ (h'+h\cot\vp)'\,\e_{\te}
\Big),
    \end{align*}
    and 
    \begin{align*}
        \div \bu=\frac{1}{\r^2}(f+g'+g\cot\vp).
    \end{align*}
\item For any $\B=\frac{B(\vp)}{\r}\,\e_\te$ with $B\in C^1([0,\pi])$, we have
    \begin{align*}
    (\bu\cdot\nb)\bu=\frac{1}{\r^3} 
    \big(
    (f'g-f^2-g^2-h^2)\,\e_{\r}
    +(g'g-h^2\cot \vp)\,\e_{\vp}
    +g(h'+h\cot\vp)\,\e_{\te}
    \Big),
    \end{align*}
    and 
    \begin{align*}
        &(\bu\cdot \nb)\B=\frac{1}{\r^3}
        \big(
        -Bh \,\e_\r-Bh \cot \vp\, \e_{\vp}+(B'g-Bf)\,\e_{\te}
        \big),\\
        &(\B\cdot \nb)\bu=\frac{1}{\r^3}
        \big(
        -Bh\, \e_\r-Bh \cot \vp\, \e_{\vp}+(Bf+Bg\cot\vp)\,\e_{\te}
        \big).
    \end{align*}
    For any $P\in C^1([0,\pi])$, the gradient of $p=\frac{P(\vp)}{\r^2}$ is
    \begin{align*}
        \nb p=\frac{1}{\r^3}(-2P\,\e_\r+P'\,\e_{\vp}).
    \end{align*}
\end{enumerate}
Here $f'=\frac{\mathrm d}{\mathrm d \vp}f(\vp)$, $f''=\frac{\mathrm d^2}{\mathrm d \vp^2}f(\vp)$.
\end{lemma}

\begin{proof}
(1) By virtue of Lemma \ref{lemmaA}, the Laplacian of $\bu$ becomes
\begin{align*}
    \Delta \bu
    =&\nb_{\e_{\r}}( \nb_{\e_{\r}} \bu)+ \nb_{\e_{\vp}} (\nb_{\e_{\vp}} \bu) + \nb_{\e_{\te}} (\nb_{\e_{\te}} \bu)    +\frac{2}{\r} \nb_{\e_{\r}} \bu + \frac1\r \cot\vp \nb_{\e_{\vp}}\bu\\
    =& \nb_{\e_{\r}}( \nb_{\e_{\r}} \bu)+ \nb_{\e_{\vp}} (\nb_{\e_{\vp}} \bu) + \nb_{\e_{\te}} (\nb_{\e_{\te}} \bu)\\ 
    &-\frac{2}{\r^3}(f\,\e_{\r}+g\, \e_{\vp} +h \,\e_{\te})
    +\frac{\cot \vp}{\r^3}\big(
    (f'-g)\,\e_{\r}+(f+g')\,\e_{\vp}+h'\,\e_{\te}
    \big).
\end{align*}
Direct calculations show
\[
\nb_{\e_{\r}}( \nb_{\e_{\r}} \bu)
=\frac{2}{\r^3}(f\,\e_{\r}+g\, \e_{\vp} +h\, \e_{\te}).
\]
Similarly, one can prove that
\[
\nb_{\e_{\vp}} (\nb_{\e_{\vp}} \bu)
=\frac{1}{\r^3}\big( (f''-2g'-f)\,\e_{\r}+(g''+2f'-g)\,\e_{\vp}+h''\,\e_{\te}
\big),
\]
and
\[
\nb_{\e_{\te}} (\nb_{\e_{\te}} \bu)
=\frac{1}{\r^3}\big(
-(f+g \cot\vp)\,\e_{\r}-(f\cot\vp+g\cot^2\vp)\, \e_{\vp}-h \csc^2 \vp\, \e_{\te}
\big).
\]
Hence, we have
\begin{align*}
\Delta \bu
=&\frac{1}{\r^3}
\Big( (f''+f'\cot\vp-2(f+g'+g\cot\vp))\,\e_\r
+(2f+g'+g\cot\vp)'\,\e_{\vp}\\
&+ (h'+h\cot\vp)'\,\e_{\te}
\Big).
\end{align*}
Similarly, equations \eqref{lemmaA1-2} and \eqref{geometricdefi} give
\begin{align*}
    \div \bu=& 
    \frac{\p u^\r}{\p \r}
    +\frac1\r \left( u^\r+ \frac{\p u^\vp}{\p \vp} \right)
    +\frac1\r \left(u^{\r}+ \cot \vp u^{\vp} + \frac{1}{\sin \vp} \frac{\p u^\te}{\p \te} \right)\\
    =& \frac{1}{\r^2}(f+g'+g\cot\vp).
\end{align*}
(2) For any vector fields $\X$ and $\Y$, by virtue of equation \eqref{lemmaA1-3}, we have $(\X\cdot \nb)\Y=\nb_\X \Y$. Using property (ii) of the Riemannian connection, the proof follows from \eqref{lemmaA1-2}. We focus on proving the case $(\bu\cdot \nb)\bu$. Direct calculation yields
\begin{align*}
   (\bu\cdot \nb)\bu=&\nb_\bu \bu
   =\frac{f}{\r}\nb_{\e_{\r}}\bu + \frac{g}{\r}\nb_{\e_{\vp}}\bu+\frac{h}{\r}\nb_{\e_{\te}}\bu\\
   =& \frac{1}{\r^3} \Big(
   (f'g-(f^2+g^2+h^2))\,\e_{\r}
   +(g'g-h^2\cot \vp)\,\e_{\vp}
   +g(h'+h\cot\vp)\,\e_{\te}
   \Big).
\end{align*}
The gradient of $p$ follows directly from \eqref{geometricdefi}.
\end{proof}

\bibliographystyle{alpha}
\bibliography{ref}

\begin{thebibliography}{ZWW25}

\bibitem[Bat99]{Batchelor99}
G.~K. Batchelor.
\newblock {\em An introduction to fluid dynamics}.
\newblock Cambridge Mathematical Library. Cambridge University Press,
  Cambridge, paperback edition, 1999.

\bibitem[CK04]{CK04}
Marco Cannone and Grzegorz Karch.
\newblock Smooth or singular solutions to the {N}avier-{S}tokes system?
\newblock {\em J. Differential Equations}, 197(2):247--274, 2004.

\bibitem[CLN06]{CLN06}
Bennett Chow, Peng Lu, and Lei Ni.
\newblock {\em Hamilton's {R}icci flow}, volume~77 of {\em Graduate Studies in
  Mathematics}.
\newblock American Mathematical Society, Providence, RI; Science Press Beijing,
  New York, 2006.

\bibitem[dC92]{dC92}
Manfredo Perdig\~ao do~Carmo.
\newblock {\em Riemannian geometry}.
\newblock Mathematics: Theory \& Applications. Birkh\"auser Boston, Inc.,
  Boston, MA, portuguese edition, 1992.

\bibitem[KMT18]{KMT18}
Kyungkeun Kang, Hideyuki Miura, and Tai-Peng Tsai.
\newblock Green tensor of the {S}tokes system and asymptotics of stationary
  {N}avier-{S}tokes flows in the half space.
\newblock {\em Adv. Math.}, 323:326--366, 2018.

\bibitem[Lan44]{Landau44}
L.~Landau.
\newblock A new exact solution of {N}avier-{S}tokes equations.
\newblock {\em C. R. (Doklady) Acad. Sci. URSS (N.S.)}, 43:286--288, 1944.

\bibitem[Squ51]{Squire51}
H.~B. Squire.
\newblock The round laminar jet.
\newblock {\em Quart. J. Mech. Appl. Math.}, 4:321--329, 1951.

\bibitem[Tsa18]{Tsai18}
Tai-Peng Tsai.
\newblock {\em Lectures on {N}avier-{S}tokes equations}, volume 192 of {\em
  Graduate Studies in Mathematics}.
\newblock American Mathematical Society, Providence, RI, 2018.

\bibitem[TX98]{TX98}
Gang Tian and Zhouping Xin.
\newblock One-point singular solutions to the {N}avier-{S}tokes equations.
\newblock {\em Topol. Methods Nonlinear Anal.}, 11(1):135--145, 1998.

\bibitem[\v{S}11]{Sverak11}
V.~\v{S}ver\'ak.
\newblock On {L}andau's solutions of the {N}avier-{S}tokes equations.
\newblock volume 179, pages 208--228. 2011.
\newblock Problems in mathematical analysis. No. 61.

\bibitem[ZWW25]{ZWW25}
Shaoheng Zhang, Kui Wang, and Yun Wang.
\newblock Point singularities of solutions to the stationary incompressible
  {MHD} equations.
\newblock {\em arXiv:2504.09607}, 2025.

\end{thebibliography}

\end{document}